\documentclass[a4paper,11pt]{article}
\usepackage[margin=1in]{geometry}  % set the margins to 1in on all sides

\usepackage{graphicx}              % to include figures
\usepackage{amsmath}
\usepackage{amscd}               % great math stuff
\usepackage{amsfonts} 
\usepackage{amssymb}             % for blackboard bold, etc
\usepackage{amsthm}                % better theorem environment

\newtheorem{theorem}{Theorem}[section]
\newtheorem{lemma}[theorem]{Lemma}
\newtheorem{cor}[theorem]{Corollary}
\newtheorem{prop}[theorem]{Proposition}

\newtheorem{remark}[theorem]{Remark}
\newtheorem{definition}[theorem]{Definition}

 \newcommand{\Cov}{{\rm Cov}}

 \newcommand{\R}{\mathbb{R}}
 \newcommand{\N}{\mathbb{N}}
 \newcommand{\Z}{\mathbb{Z}}

 \def\1{{\mathchoice {1\mskip-4mu\mathrm l} 
{1\mskip-4mu\mathrm l}
{1\mskip-4.5mu\mathrm l} {1\mskip-5mu\mathrm l}}}

 \newcommand{\Dcal}{{\mathcal D}}

 \newcommand{\lfl}{\left\lfloor}
 \newcommand{\rfl}{\right\rfloor}
 
\newcommand{\X}{{\bf X}}
\newcommand{\Y}{{\bf Y}}

\renewcommand{\subsection}{\secdef \subsct\sbsect}
\newcommand{\subsct}[2][default]{\refstepcounter{subsection}
\vspace{0.15cm}
{\flushleft\bf \arabic{section}.\arabic{subsection}~\bf #1  }
\nopagebreak\nopagebreak}
\newcommand{\sbsect}[1]{\vspace{0.1cm}\noindent
{\bf #1}\vspace{0.1cm}}

\newenvironment{example}{\refstepcounter{theorem}
{\bf Example \thetheorem\ }\nopagebreak  }%
{\nopagebreak {\hfill\rule{2mm}{2mm}}\\ }

\newtheoremstyle{thm}{1.5ex}{1.5ex}{\itshape\rmfamily}{}
{\bfseries\rmfamily}{}{2ex}{}

\newtheoremstyle{rem}{1.3ex}{1.3ex}{\rmfamily}{}
{\itshape\rmfamily}{}{1.5ex}{}
\theoremstyle{rem}
\refstepcounter{subsubsection}

\def\thebibliography#1{\section*{References}
  \list%
  {\arabic{enumi}.}%                          {\star}{\star}{\star} style of reference number {\star}{\star}{\star}
    {\settowidth\labelwidth{[#1]}\leftmargin\labelwidth
    \advance\leftmargin\labelsep
    \parsep0pt\itemsep0pt
    \usecounter{enumi}}
    \def\newblock{\hskip .11em plus .33em minus .07em}
    \sloppy                   % \clubpenalty4000\widowpenalty4000
    \sfcode`\.=1000\relax}

\begin{document}

\title{A simple proof of distance bounds for Gaussian rough paths}

\author{\textsc{Sebastian Riedel and Weijun Xu} \\ 
\textit{Technische Universit\"{a}t Berlin and University of Oxford}}

\maketitle

\abstract{We derive explicit distance bounds for Stratonovich iterated integrals along two Gaussian processes (also known as signatures of Gaussian rough paths) based on the regularity assumption of their covariance functions. Similar estimates have been obtained recently in [Friz-Riedel, AIHP, to appear]. One advantage of our argument is that we obtain the bound for the third level iterated integrals merely based on the first two levels, and this reflects the intrinsic nature of rough paths. Our estimates are sharp when both covariance functions have finite $1$-variation, which includes a large class of Gaussian processes. 

Two applications of our estimates are discussed. The first one gives the a.s. convergence rates for approximated solutions to rough differential equations driven by Gaussian processes. In the second example, we show how to recover the optimal time regularity for solutions of some rough SPDEs. }

\begin{flushleft}

\smallskip

\textbf{AMS classification:} 60
\end{flushleft}

\bigskip

\section{Introduction}

The intersection between rough path theory and Gaussian processes has been an active research area in recent years (\cite{FV10G}, \cite{FV10}, \cite{H11}). The central idea of rough paths, as realized by Lyons (\cite{L98}), is that the key properties needed for defining integration against an irregular path do not only come from the path itself, but from the path together with a sequence of iterated integrals along the path, namely
\begin{align} \label{iterated integrals}
	\X^{n}_{s,t} = \int_{s < u_{1} < \cdots < u_{n} < t} dX_{u_{1}} \otimes \ldots \otimes dX_{u_{n}}. 
\end{align}
In particular, Lyons extension theorem shows that for paths of finite $p$-variation, the first $\lfl p \rfl$ levels iterated integrals determine all higher levels. For instance, if $p=1$, the path has bounded variation and the higher iterated integrals coincide with the usual Riemann-Stieltjes integrals. However, for $p\geq 2$, this is not true anymore and one has to say what the second (and possibly higher) order iterated integrals should be before they determine the whole rough path.

Lyons and Zeitouni (\cite{LZ98}) were the first to study iterated Wiener integrals in the sense of rough paths. They provide sharp exponential bounds on the iterated integrals of all levels by controlling the variation norm of the L\'{e}vy area. The case of more general Gaussian processes were studied by Friz and Victoir in \cite{FV10G} and \cite{FV10}. They showed that if $X$ is a Gaussian process with covariance of finite $\rho$-variation for some $\rho \in [1,2)$, then its iterated integrals in the sense of \eqref{iterated integrals} can be defined in a natural way and we can lift $X$ to a Gaussian rough path $\X$. 

In the recent work \cite{FR11b}, Friz and the first author compared the two lift maps $\X$ and $\Y$ for the joint process $(X, Y)$. It was shown that their average distance in rough paths topology can be controlled by the value $\sup_{t}|X_{t} - Y_{t}|_{L^2}^{\zeta}$ for some $\zeta >0$, and a sharp quantitative estimate for $\zeta$ was given. In particular, it was shown that considering both rough paths in a larger rough paths space (and therefore in a different topology) allows for larger choices of $\zeta$. Using this, the authors derived essentially optimal convergence rates for $\X^{\epsilon} \to \X$ in rough paths topology when $\epsilon \to 0$ where $\X^{\epsilon}$ is a suitable approximation of $\X$. 

In order to prove this result, sharp estimates of $|\X_{s,t}^{n} - \Y_{s,t}^{n}|$ need to be calculated on every level $n$. Under the assumption $\rho \in [1,\frac{3}{2})$, the sample paths of $\X$ and $\Y$ are $p$-rough paths for any $p > 2\rho$, hence we can always choose $p<3$ and therefore the first two levels determine the entire rough path. Lyons' continuity theorem then suggests that one only needs to give sharp estimates on level 1 and 2; the estimates on the higher levels can be obtained from the lower levels through induction. On the other hand, interestingly, one additional level was estimated "by hand" in \cite{FR11b} before performing the induction. To understand the necessity of computing this additional term, let us note from \cite{L98} that the standard distance for two deterministic $p$-rough paths takes the form of the smallest constant $C_n$ such that
\begin{align*}
|\X_{s,t}^{n} - \Y_{s,t}^{n}| \leq C_{n} \epsilon \omega(s,t)^{\frac{n}{p}}, \qquad n = 1, \cdots, \lfl p \rfl
\end{align*}
holds for all $s<t$ where $\omega$ is a control function to be defined later. The exponent on the control for the next level is expected to be 
\begin{align} \label{condition on the exponent}
\frac{n+1}{p} = \frac{\lfl p \rfl + 1}{p} > 1, 
\end{align}
so when one repeats Young's trick of dropping points in the induction argument (the key idea of the extension theorem), condition \eqref{condition on the exponent} will ensure that one can establish a maximal inequality for the next level. However, in the current problem where Gaussian randomness is involved, the $L^{2}$ distance for the first $\lfl 2 \rho \rfl$ iterated integrals takes the form
\begin{align*}
|\X_{s,t}^{n} - \Y_{s,t}^{n}|_{L^{2}} < C_{n} \epsilon \omega(s,t)^{\frac{1}{2 \gamma} + \frac{n-1}{2 \rho}}, \qquad n = 1, 2, \qquad \rho \in [1, \frac{3}{2}), 
\end{align*}
where $\gamma$ might be much larger than $\rho$. Thus, the '$n-1$' in the exponent leaves condition \eqref{condition on the exponent} unsatisfied, and one needs to compute the third level by hand before starting induction on $n$. 

In this article, we resolve the difficulty by moving part of $\epsilon$ to fill in the gap in the control so that the exponent for the third level control reaches $1$. In this way, we obtain the third level estimate merely based on the first two levels, and it takes the form
\begin{align*}
|\X_{s,t}^{3} - \Y_{s,t}^{3}|_{L^{2}} < C_{3} \epsilon^{\eta} \omega(s,t), 
\end{align*}
where $\eta \in (0,1]$, and its exact value depends on $\gamma$ and $\rho$. We see that there is a $1 - \eta$ reduction in the exponent of $\epsilon$, which is due to the fact that it is used to compensate the control exponent. This interplay between the 'rate' exponent and the control exponent can be viewed as an analogy to the relationship between time and space regularities for solutions to SPDEs. We will make the above heuristic argument rigorous in section 4. We also refer to the recent work \cite{LX11} for the situation of deterministic rough paths.

Our main theorem is the following.

\bigskip

\begin{theorem} \label{main theorem}
Let $(X,Y) = (X^{1},Y^{1}, \cdots, X^{d}, Y^{d}): [0,T] \rightarrow \R^{d+d}$ be a centered Gaussian process on a probability space $(\Omega,\mathcal{F},P)$ where $(X^i,Y^i)$ and $(X^j,Y^j)$ are independent for $i \neq j$, with continuous sample paths and covariance function $R_{(X,Y)}\colon [0,T]^2 \mapsto \R^{2d\times 2d}$. Assume further that there is a $\rho \in [1, \frac{3}{2})$ such that the $\rho$-variation of $R_{(X,Y)}$ is bounded by a finite constant $K$. Let $\gamma \geq \rho$ such that $\frac{1}{\gamma} + \frac{1}{\rho} > 1$. Then, for every $\sigma > 2 \gamma$, $N \geq \lfl \sigma \rfl$, $q \geq 1$ and every $\delta > 0$ small enough, there exists a constant $C = C(\rho,\gamma,\sigma,K,\delta,q)$ such that
\begin{enumerate}

\item If $\frac{1}{2\gamma} + \frac{1}{\rho} > 1$, then
\begin{align*}
|\varrho_{\sigma-\text{var}}^{N}(\X,\Y)|_{L^{q}} \leq C \sup_{t \in [0,T]} |X_{t} - Y_{t}|_{L^{2}}^{1 - \frac{\rho}{\gamma}}. 
\end{align*}

\item If $\frac{1}{2\gamma} + \frac{1}{\rho} \leq 1$, then
\begin{align*}
|\varrho_{\sigma-\text{var}}^{N}(\X,\Y)|_{L^{q}} \leq C \sup_{t \in [0,T]} |X_{t} - Y_{t}|_{L^{2}}^{3 - 2 \rho - \delta}. 
\end{align*}

\end{enumerate}
\end{theorem}

\bigskip

The proof of this theorem will be postponed to section 4.2 after we have established all the estimates needed. We first give two remarks.

\bigskip

\begin{remark}
We emphasize that the constant $C$ in the above theorem depends on the process $(X,Y)$ \textit{only} through the parameters $\rho$ and $K$. 
\end{remark}

\bigskip

\begin{remark}
For $N = \lfl \sigma \rfl$, $\varrho_{\sigma-\text{var}}^{N}$ denotes an inhomogeneous rough paths metric. See section \ref{sec_rp_essentials} for the precise definition.
\end{remark}

\bigskip

Our paper is structured as follows. In section 2, we provide some important concepts and notations from rough path theory that are necessary for our problem. In section 3, we introduce the class of Gaussian processes which possess a lift to Gaussian rough paths and estimate the difference of two Gaussian rough paths on level one and two. Section 4 is devoted to the proof of the main theorem. We first obtain the third level estimate directly from the first two levels, which requires a technical extension of Lyons' continuity theorem, and justify the heuristic argument above rigorously. All higher level estimates are then obtained with the induction procedure in \cite{L98}, and the claim of the main theorem follows. In section 5, we give two applications of our main theorem. The first one deals with convergence rates for Wong-Zakai approximations in the context of rough differential equations. The second example shows how to derive optimal time regularity for the solution of a modified stochastic heat equation seen as an evolution in rough paths space.

\bigskip

\textsc{Notations.} Throughout the paper, $C, C_{n}, C_{n}(\rho, \gamma)$ will denote constants depending on certain parameters only, and their actual values may change from line to line.

\bigskip

\textsc{Acknowledgements.} We wish to thank our advisors, Peter Friz and Terry Lyons, for their helpful discussions and support during the project. S.Riedel is supported by a PhD scholarship from the Berlin Mathematical School (BMS). W.Xu is supported by the Oxford-Man Institute Scholarship.

\bigskip

\bigskip

\section{Elements from Rough path theory}\label{sec_rp_essentials}

In this section, we introduce the concepts and definitions from rough path theory that are necessary for our current application. For a detailed account of the theory, we refer readers to \cite{FV10}, \cite{LCL06} and \cite{LQ02}. 

Fix the time interval $[0,T]$. For all $s < t \in [0,T]$, let $\Delta_{s,t}$ denote the simplex
\begin{align*}
\{ (u_{1}, u_{2})\ |\ s \leq u_{1} \leq u_{2} \leq t \}, 
\end{align*}
and we simply write $\Delta$ for $\Delta_{0,T}$. In what follows, we will use $x$ to denote an $\R^{d}$-valued path, and $X$ to denote a stochastic process in $\R^{d}$, which is a Gaussian process in the current paper. For any integer $N$, let
\begin{align*}
T^{N}(\R^{d}) = \R \oplus \R^{d} \oplus \cdots \oplus (\R^{d})^{\otimes N} 
\end{align*}
denote the truncated tensor algebra. The space of all continuous bounded variation paths $x\colon [0,T] \to \R^d$ is denoted by $C^{1-var}(\R^{d})$. For a path $x \in C^{1-var}(\R^{d})$, we use the bold letter $\textbf{X}$ to denote its $n$-th level iterated tensor integral: 
\begin{align*}
\X_{s,t}^{n} = \int_{s < u_{1} < \cdots < u_{n} < t} dx_{u_{1}} \otimes \cdots \otimes dx_{u_{n}}. 
\end{align*}
The lift map $S_{N}$ taking $x$ to a $T^{N}(\R^{d})$-valued path is defined by
\begin{align*}
S_{N}(x)_{s,t} = 1 + \sum_{n=1}^{N} \X_{s,t}^{n}. 
\end{align*}
For a path $x$, write $x_{s,t} = x_{t} - x_{s}$ and we have $x_{s,t} = \textbf{X}_{s,t}^{1}$. It is well known that $S_{N}(x)$ is a multiplicative functional, that is, for any $s < u < t$, we have
\begin{align*}
S_{N}(x)_{s,u} \otimes S_{N}(x)_{u,t} = S_{N}(x)_{s,t}, 
\end{align*}
where the multiplication takes place in $T^{N}(\R^{d})$.

For each subspace $(\R^{d})^{\otimes n}$, there is an associated tensor norm $|\cdot|$. If $\X, \Y$ are two multiplicative functionals in $T^N(\R^d)$, then for each $p \geq 1$, we define their $p$-variation distance by
\begin{align*}
\varrho_{p-var}^{N}(\X, \Y):= \max_{n \leq N} \sup_{\mathcal{P}} \big(\sum_{i} |\X_{t_{i}, t_{i+1}}^{n} - \Y_{t_{i}, t_{i+1}}^{n}|^{\frac{p}{n}} \big)^{\frac{n}{p}}, 
\end{align*}
where the supremum is taken over all finite partitions of the interval $[0,T]$. If $N= \lfl p \rfl$, this defines a rough paths metric and we only write $\varrho_{p-var}(\X, \Y)$ in this case.

\bigskip

\begin{remark}
Note that we only consider so-called \emph{inhomogeneous} rough paths metrics in this paper. The reason for this is that the It\={o}-Lyons solution map for rough paths is locally Lipschitz with respect to these metrics (cf. \cite[Chapter 10]{FV10}).
\end{remark}

\bigskip

We define the subset $G^{N}(\R^{d}) \subset T^{N}(\R^{d})$ to be
\begin{align*}
G^{N}(\R^{d}) = \{ S_{N}(x)_{0,1}: x \in C^{1-var}(\R^{d}) \}. 
\end{align*}
The multiplicativity of $S_{N}$ implies that $G^{N}(\R^{d})$ is a group with multiplication $\otimes$ and identity element $1$. If $x \in C^{1-var}(\R^{d})$, one can shows that actually $S_{N}(x)_{s,t} \in G^{N}(\R^{d})$ for all $s<t$.
%Finally, we use $C^{0, p-var}([0,T], G^{N}(\R^{d}))$ to denote the closure of $\{S_{N}(x)_{0,T}\}$ under the metric $\varrho_{p-var}^{N}$, where $x$ is taken over all smooth paths in $\R^{d}$, running through the time interval $[0,T]$.  

\bigskip

\begin{definition}
A function $\omega: \Delta \rightarrow \R^{+}$ is called a \emph{control} if it is continuous, vanishes on the diagonal, and is superadditive in the sense that for any $s < u < t$, we have
\begin{align*}
\omega(s,u) + \omega(u,t) \leq \omega(s,t). 
\end{align*}
\end{definition}

\bigskip

We say a multiplicative functional $\X$ in $T^{N}(\R^{d})$ has finite $p$-variation ($p \geq 1$) controlled by $\omega$ if for each $n \leq N$, there exists a constant $C_{n}$ such that for all $s < t$, we have
\begin{align*}
|\X_{s,t}^{n}| \leq C_{n} \omega(s,t)^{\frac{n}{p}}. 
\end{align*}

\bigskip

\begin{definition}
Let $p \geq 1$. A \emph{geometric $p$-rough path} is a continuous path in $G^{\lfl p \rfl}(\R^{d})$ which is in the $p$-variation closure (w.r.t the metric $\varrho_{p-var}$) of the set of bounded variation paths. We use $C^{0, p-var}([0,T], G^{\lfl p \rfl}(\R^{d}))$ to denote the space of geometric $p$-rough paths.
\end{definition}

\bigskip

By Lyon's extension theorem (cf. \cite[Theorem 2.2.1]{L98}), every (geometric) $p$-rough path $\X$ can be lifted to a $q$-rough path for every $q \geq p$. Abusing notation, we will use the same letter $\X$ to denote this lift. We will write $\X^n = \pi_n(\X)$ where $\pi_n$ denotes the projection of $T^N(\R^d)$ onto the $n$-th tensor product, $n\leq N$. If $x = \pi_1(\X)$, we will also use the notation
	\begin{align*}
		\X^n_{s,t} = \int_{\Delta_{s,t}^n} dx \otimes \cdots \otimes dx
	\end{align*}
(even though this integral does not have to exist as a limit of Riemann sums). 
\bigskip

\bigskip

\section{2D variation and Gaussian rough paths}

If $I=[a,b]$ is an interval, a \emph{dissection of $I$} is a finite subset of points of the form $\{a=t_0< \ldots < t_m =b\}$. The family of all dissections of $I$ is denoted by $\Dcal (I)$. 

Let $I \subset \R$ be an interval and $A=[a,b]\times [c,d]\subset I \times I$ be a rectangle. 
%If $a=b$ or $c=d$, we call $A$ \emph{degenerate}. Two rectangles are called \emph{essentially disjoint} if their intersection is empty or degenerate. A \emph{partition} of $A$ is a finite set of essentially disjoint rectangles whose union is $A$. The family of all partitions of $A$ is denoted by $\Pcal(A)$. 
If $f\colon I\times I \to V$ is a function, mapping into a normed vector space $V$, we define the \emph{rectangular increment} $f(A)$ by setting

\begin{align*}
	f(A) := f\left( \begin{array}{c} a,b \\ c,d \end{array} \right) 
	:=	f\left( \begin{array}{c} b \\ d \end{array} \right)
	- f\left( \begin{array}{c} a \\ d \end{array} \right)
	- f\left( \begin{array}{c} b \\ c \end{array} \right)
	+ f\left( \begin{array}{c} a \\ c \end{array} \right).
\end{align*}

\bigskip

\begin{definition}
	Let $p\geq 1$ and $f\colon I \times I \to V$. For $[s,t]\times[u,v] \subset I \times I$, set
	\begin{align*}
		V_{p}(f;[s,t]\times[u,v]) := \left( \sup_{\substack{(t_i) \in \Dcal([s,t]) \\ (t'_j) \in \Dcal([u,v])}} \sum_{t_i,t'_j}
		\left| f\left( \begin{array}{c} t_i,t_{i+1} \\ t'_j,t'_{j+1} \end{array}\right) \right|^p
		\right)^{\frac{1}{p}}.
	\end{align*}
	If $V_{p}(f,I\times I)<\infty$, we say that $f$ has finite ($2D$) $p$-variation. We also define
	\begin{align*}
		V_{\infty}(f;[s,t]\times[u,v]) := \sup_{\substack{ \sigma, \tau \in [s,t] \\ \mu, \nu \in [u,v]}} \left| f\left( \begin{array}{c} \sigma, \tau \\ \mu, \nu \end{array} \right) \right|
	\end{align*}
\end{definition}

%\bigskip
%
%
%
%\begin{flushleft}
%\begin{example}
%	Let $f,g\colon I\to \R$ be functions of finite $p$-variation and define $f\otimes g\colon I\times I \to \R$, $(s,t)\mapsto f(s)g(t)$. For $A=[a,b]\times [c,d]$ we have $(f\otimes g)(A) = (f(b) - f(a))(g(d) - g(c))$ and therefore
%	\begin{align*}
%		V_{p-\text{var}}(f\otimes g;[s,t]\times [u,v]) \leq |f|_{p-\text{var};[s,t]} |g|_{p-\text{var};[u,v]}.
%	\end{align*}
%	In particular, $f\otimes g$ has finite $2D$ $p$-variation.
%\end{example}
%\end{flushleft}

%Let $1 \leq p \leq p'$ and $f\colon I\times I\to \R$. In $\R^n$, we know that $|\cdot|_{l^{p'}} \leq |\cdot|_{l^p}$ from which we can deduce that
%\begin{align*}
%	V_{p'-\text{var}}(f ;I\times I) \leq V_{p-\text{var}}(f ;I\times I).
%\end{align*}

\bigskip

\begin{lemma}\label{lemma:infinterpolation}
	Let $f\colon I\times I\to V$ be a continuous map and $1\leq p \leq p' < \infty$. Assume that $f$ has finite $p$-variation. Then for every $[s,t]\times[u,v] \subset I\times I$ we have
		\begin{align*}
			V_{p'}(f;[s,t]\times [u,v]) \leq V_{\infty}(f;[s,t]\times [u,v])^{1-\frac{p}{p'}}\, V_{p}(f;[s,t]\times [u,v])^{\frac{p}{p'}}
		\end{align*}

%  the following inequalities hold:
%	\begin{itemize}
%		\item[(i)] For every $[s,t]\times[u,v] \subset I\times I$,
%		\begin{align*}
%			V_{\infty}(f ;[s,t]\times[u,v]) \leq V_{p_1-\text{var}}(f ;[s,t]\times[u,v]) \leq V_{p_0-\text{var}}(f ;[s,t]\times[u,v]).
%		\end{align*}
%		\item[(ii)]
%		For $0\leq \theta \leq 1$, set $\frac{1}{p_{\theta}}=\frac{1-\theta}{p_0} + \frac{\theta}{p_1}$. Then
%		\begin{align*}
%			V_{p_{\theta}-\text{var}}(f ;[s,t]\times[u,v]) \leq V_{p_0-\text{var}}(f ;[s,t]\times[u,v])^{1-\theta} V_{p_1-\text{var}}(f ;[s,t]\times[u,v])^{\theta}
%		\end{align*}
%		holds for every $[s,t]\times[u,v] \subset I\times I$.
%		\item[(iii)]
%		For every $[s,t]\times[u,v] \subset I\times I$
%		\begin{align*}
%			V_{p_1-\text{var}}(f;[s,t]\times [u,v]) \leq V_{\infty}(f;[s,t]\times [u,v])^{1-\frac{p_0}{p_1}}V_{p_0-\text{var}}(f;[s,t]\times [u,v])^{\frac{p_0}{p_1}}
%		\end{align*}
%	\end{itemize}
\end{lemma}

\begin{proof}
%	The first inequality in (i) follows by definition, the second one by the inequality $|\cdot|_{l^{p_1}} \leq |\cdot|_{l^{p_0}}$ which hold in $\R^n$ for every $n\in\N$. 
	Let $(t_i) \in \Dcal([s,t])$ and $(t'_j) \in \Dcal([u,v])$. Then,
	\begin{align*}
		\sum_{t_i,t'_j}	\left| f\left( \begin{array}{c} t_i,t_{i+1} \\ t'_j,t'_{j+1} \end{array}\right) \right|^{p'} \leq V_{\infty}(f;[s,t]\times[u,v])^{p'-p} \sum_{t_i,t'_j}	\left| f\left( \begin{array}{c} t_i,t_{i+1} \\ t'_j,t'_{j+1} \end{array}\right) \right|^{p}.
	\end{align*}
	Taking the supremum over all partitions gives the claim.
	%(ii) follows similarly using H\"older's inequality in $\R^n$.
\end{proof}

\bigskip
%
%
%
%\begin{remark}
%	The Lemma shows in particular that
%	\begin{align*}
%		V_{\infty}(f ;[s,t]\times[u,v]) \leq \liminf_{p\to\infty} V_{p-\text{var}}(f ;[s,t]\times[u,v]).
%	\end{align*}
%	On the other hand, from (iii) we can deduce that
%	\begin{align*}
%		\limsup_{p_1 \to \infty} V_{p_1-\text{var}}(f;[s,t]\times [u,v]) &\leq \limsup_{p_1 \to \infty} V_{\infty}(f;[s,t]\times [u,v])^{1-\frac{p_0}{p_1}}V_{p_0-\text{var}}(f;[s,t]\times [u,v])^{\frac{p_0}{p_1}} \\
%		&= V_{\infty}(f;[s,t]\times [u,v]),
%	\end{align*}
%	hence
%	\begin{align*}
%		\lim_{p\to\infty} V_{p-\text{var}}(f;[s,t]\times [u,v]) = V_{\infty}(f;[s,t]\times [u,v])		
%	\end{align*}
%	which justifies our notation.
%\end{remark}
%
%
%
%\bigskip

%Set $\Delta_I = \{(s,t) \in I\times I\ :\ s<t \}$. Recall the definition of a control $\omega$.
%
%
%\bigskip

\begin{lemma}\label{lemma:exist_control}
	Let $f\colon I\times I \to \R$ be continuous with finite $p$-variation. Choose $p'$ such that $p' \geq p$ if $p=1$ and $p' > p$ if $p>1$. Then there is a control $\omega$ and a constant $C=C(p,p')$ such that
	\begin{align*}
		V_{p'}(f;J\times J) \leq \omega(J)^{\frac{1}{p'}} \leq C V_{p}(f;J\times J)
	\end{align*}
	holds for every interval $J\subset I$.
\end{lemma}

\begin{proof}
	Follows from \cite[Theorem 1]{FV11}.
\end{proof}

%
%For $s<t$, set 
%\begin{align*}
%	\Delta_{s,t}^n := \{(u_1,\ldots,u_n)\subset [s,t]^n\ :\ s< u_1 <\ldots <u_n <t \}
%\end{align*} 
%for the $n$-simplex. We also define $\Delta_I := \Delta_I^2$.
%
%
%\begin{definition}
%	A map $\omega\colon \Delta_I \times \Delta_I \to [0,\infty)$ is a \emph{$2D$ control} if it is continuous, zero on degenerate rectangles and super-additive in the sense that, for all rectangles $A\subset I\times I$,
%	\begin{align*}
%		\sum_{i=1}^n \omega(A_i) \leq \omega(A)
%	\end{align*}
%	whenever $\{A_1,\ldots,A_n\} \in \Pcal(A)$. If for a function $f\colon I\times I \to \R$ there exists a control $\omega$ such that for every $A\subset I\times I$ one has $|f(A)|^p \leq \omega(A)$, we say that the $p$-variation of $f$ is controlled by $\omega$.
%\end{definition}
%
%Clearly, if the $p$-variation of a function is controlled by $\omega$, it has finite $p$-variation. The converse is not true in general (cf. [Friz, Victoir, Note...]), but in the important case $p=1$ it is, as the following lemma shows:
%
%\begin{lemma}
%	If $f\colon I\times I \to \R$ has finite $1$-variation, the map
%	\begin{align*}
%		[s,t] \times [u,v] \mapsto V_{1-\text{var}}(f;[s,t]\times[u,v])
%	\end{align*}
%	defines a $2D$ control. In particular, the $1$-variation is controlled by this control function.
%\end{lemma}
%
%\begin{proof}
%	Follows from [Friz, Victoir, Note...,], Theorem 1.
%\end{proof}

\bigskip

Let $X = (X^1,\ldots, X^d) \colon I \to \R^d$ be a centered, stochastic process. Then the covariance function $R_X(s,t) := \Cov_X(s,t) = E(X_s \otimes X_t)$ is a map $R_X \colon I\times I \to \R^{d\times d}$ and we can ask for its $\rho$-variation (we will use the letter $\rho$ instead of $p$ in this context). Clearly, $R_X$ has finite $\rho$-variation if and only if for every $i,j \in \{1,\ldots,d\}$ the map $s,t \mapsto E(X^i_s X^j_t)$ has finite $\rho$-variation. In particular, if $X^i$ and $X^j$ are independent for $i\neq j$, $R_X$ has finite $\rho$-variation if and only if $R_{X^i}$ has finite $\rho$-variation for every $i=1,\ldots,d$. In the next example, we calculate the $\rho$-variation for the covariances of some well-known real valued Gaussian processes. In particular, we will see that many interesting Gaussian processes have a covariance of finite $1$-variation. 

\bigskip

\begin{flushleft}
\begin{example}
	\begin{enumerate}
		\item Let $X=B$ be a Brownian motion. Then $R_B(s,t)=\min\{s,t\}$ and thus, for $A=[s,t]\times [u,v]$,
		\begin{align*}
			\left| R(A)\right| = \left| (s,t)\cap (u,v)\right| = \int_{[s,t]\times[u,v]} \delta_{x=y}\, dx\, dy. 
		\end{align*}
		This shows that $R_B$ has finite $1$-variation on any interval $I$.
		
		\item More generally, let $f\colon [0,T]\to \R$ be a left-continuous, locally bounded function. Set
		\begin{align*}
			X_t = \int_0^t f(r)\,dB_r.
		\end{align*}
		Then, for $A=[s,t]\cap[u,v]$ we have by the It\={o} isometry,
		\begin{align*}
			R_X(A) = E\left[\int_{[s,t]} f \,dB \int_{[u,v]} f\,dB \right] = \int_{[s,t]\times [u,v]} \delta_{x=y} f(x) f(y)\,dx\,dy
		\end{align*}
		which shows that $R_X$ has finite $1$-variation.		
		
		\item Let $X$ be an Ornstein-Uhlenbeck process, i.e. $X$ is the solution of the SDE
		\begin{align}\label{eqn_OU-SDE}
			dX_t = -\theta X_t\, dt + \sigma\,dB_t %, \quad X_0 =0
		\end{align}
		for some $\theta, \sigma >0$. If we claim that $X_0 =0$, one can show that $X$ is centered, Gaussian and a direct calculation shows that the covariance of $X$ has finite $1$-variation on any interval $[0,T]$. The same is true considering the stationary solution of \eqref{eqn_OU-SDE} instead.
		%if we look for the stationary solution of 
		
		\item If $X$ is a continuous Gaussian martingale, it can be written as a time-changed Brownian motion. Since the $\rho$-variation of its covariance is invariant under time-change, $X$ has again a covariance of finite $1$-variation.
		
		\item If $X\colon [0,T]\to\R$ is centered Gaussian with $X_0 = 0$, we can define a Gaussian bridge by
		\begin{align*}
			X_{\text{Bridge}}(t) = X_t - t \frac{X_T}{T}.
		\end{align*}
		One can easily show that if the covariance of $X$ has finite $\rho$-variation, the same is true for $X_{\text{Bridge}}$. In particular, Brownian bridges have finite $1$-variation.

	\end{enumerate}
\end{example}
\end{flushleft}

Next, we cite the fundamental existence result about Gaussian rough paths. For a proof, cf. \cite{FV10G} or \cite[Chapter 15]{FV10}.

\begin{theorem}[Friz, Victoir]
	Let $X\colon [0,T] \to \R^d$ be a centered Gaussian process with continuous sample paths and independent components. Assume that there is a $\rho \in [1,2)$ such that $V_{\rho}(R_X ; [0,T]^2) < \infty$. Then $X$ admits a lift $\X$ to a process whose sample paths are geometric $p$-rough paths for any $p>2\rho$, i.e. with sample paths in $C^{0,p-\text{var}}([0,T],G^{\lfloor p \rfloor }(\R^d))$ and $\pi_1(\X_{s,t}) = X_t - X_s$ for any $s<t$.  
\end{theorem}

\bigskip

In the next proposition, we give an upper $L^2$-estimate for the difference of two Gaussian rough paths on the first two levels.

\bigskip

\begin{prop}\label{prop_estimates_first_levels}
	Let $(X,Y) = (X^{1},Y^{1}, \cdots, X^{d}, Y^{d}): [0,T] \rightarrow \R^{d+d}$ be a centered Gaussian process with continuous sample paths where $(X^i,Y^i)$ and $(X^j,Y^j)$ are independent for $i \neq j$. Let $\rho \in [1,\frac{3}{2})$ and assume that $V_{\rho'}(R_{(X,Y)}, [0,T]^{2}) \leq K < +\infty$ for a constant $K>0$ where $\rho'<\rho$ in the case $\rho>1$ and $\rho'=1$ in the case $\rho=1$. Let $\gamma \geq \rho$ such that $\frac{1}{\gamma} + \frac{1}{\rho} > 1$. Then there are constants $C_0,C_1,C_2$ dependending on $\rho,\rho',\gamma$ and $K$ and a control $\omega$ such that $\omega(0,T) \leq C_0$ and 
	\begin{align*}
		\left| X_{s,t} - Y_{s,t} \right|_{L^2} \leq C_1 \sup_{u\in[s,t]} \left| X_u - Y_u \right|_{L^2}^{1-\frac{\rho}{\gamma}} \omega(s,t)^{\frac{1}{2\gamma}}
	\end{align*}
	and
	\begin{align*}
	\left| \int_s^t X_{s,u} \,\otimes dX_u - \int_s^t Y_{s,u} \,\otimes dY_u\right|_{L^2} 
	\leq C_2 \sup_{u\in[s,t]} \left| X_u - Y_u \right|_{L^2}^{1-\frac{\rho}{\gamma}} \omega(s,t)^{\frac{1}{2\gamma} + \frac{1}{2\rho}}
	\end{align*}
	hold for every $s<t$.
\end{prop}

\begin{proof}
	Note first that, by assumption on $V_{\rho'}(R_{(X,Y)};[0,T]^2)$, Lemma \ref{lemma:exist_control} guarantees that there is a control $\omega$ and a constant $c_1=c_1(\rho,\rho')$ such that
	\begin{align*}
		V_{\rho}(R_{X};[s,t]^2) \vee V_{\rho}(R_{Y};[s,t]^2) \vee V_{\rho}(R_{(X-Y)};[s,t]^2) \leq \omega(s,t)^{1/{\rho}}  
	\end{align*}
	holds for all $s<t$ and $i=1,\ldots,d$ with the property that $\omega(0,T) \leq c_1 K^{\rho} =:C_0$. We will estimate both levels componentwise. We start with the first level. Let $i\in\{1,\ldots,d\}$. Then,
	\begin{align*}
		\left| X^i_{s,t} - Y^i_{s,t} \right|_{L^2}^2 &= \left| R_{(X^i - Y^i)}\left( \begin{array}{c} s,t \\ s,t \end{array} \right) \right| \\
		&\leq V_{\gamma}(R_{(X - Y)};[s,t]^2) \\
	\end{align*}
	and thus
	\begin{align*}
		\left| X_{s,t} - Y_{s,t} \right|_{L^2} \leq c_2 \sqrt{V_{\gamma}(R_{(X - Y)};[s,t]^2)}.
	\end{align*}
	For the second level, consider first the case $i=j$. We have, using that $(X,Y)$ is Gaussian and that we are dealing with \emph{geometric} rough paths,
	\begin{align*}
		\left| \int_s^t X^i_{s,u} \,dX^i_u - \int_s^t Y^i_{s,u} \,dY^i_u\right|_{L^2}
		&= \frac{1}{2} \left| (X_{s,t}^i)^2 - (Y_{s,t}^i)^2 \right|_{L^2} \\
		&= \frac{1}{2} \left| (X_{s,t}^i - Y_{s,t}^i)(X_{s,t}^i + Y_{s,t}^i) \right|_{L^2} \\
		&\leq c_3 \left| X_{s,t}^i - Y_{s,t}^i \right|_{L^2} \left(|X_{s,t}^i|_{L^2} + |Y_{s,t}^i|_{L^2} \right).
	\end{align*}
	From the first part, we know that
	\begin{align*}
		\left| X_{s,t}^i - Y_{s,t}^i \right|_{L^2} \leq \sqrt{V_{\gamma}(R_{(X - Y)};[s,t]^2)}.
	\end{align*}
	Furthermore,
	\begin{align*}
		|X_{s,t}^i|_{L^2}	= \sqrt{\left| R_{X}\left( \begin{array}{c} s,t \\ s,t \end{array} \right) \right|}
		\leq \sqrt{V_{\rho}(R_{X};[s,t]^2)} \leq \omega(s,t)^{\frac{1}{2\rho}}
	\end{align*}
	and the same holds for $|Y_{s,t}^i|_{L^2}$. Hence
	\begin{align*}
		\left| \int_s^t X^i_{s,u} \,dX^i_u - \int_s^t Y^i_{s,u} \,dY^i_u\right|_{L^2} \leq c_4 \sqrt{V_{\gamma}(R_{(X - Y)};[s,t]^2)}\omega(s,t)^{\frac{1}{2\rho}}.
	\end{align*}
	For $i\neq j$,
	\begin{align*}
		&\left| \int_s^t X^i_{s,u} \,dX^j_u - \int_s^t Y^i_{s,u} \,dY^j_u\right|_{L^2} \\
		\leq &\left| \int_s^t (X^i - Y^i)_{s,u} \,dX^j_u \right|_{L^2}
		+ \left| \int_s^t Y^i_{s,u} \,d(X^j - Y^j)_u \right|_{L^2}.
	\end{align*}
	We estimate the first term. From independence,
	\begin{align*}
		E \left[\left(\int_s^t (X^i - Y^i)_{s,u} \,dX^j_u\right)^2 \right]
		= \int_{[s,t]^2} R_{(X^i - Y^i)}\left( \begin{array}{c} s,u \\ s,v \end{array} \right) \, dR_{X^j}(u,v)
	\end{align*}
	where the integral on the right is a $2D$ Young integral.\footnote{The reader might feel a bit uncomfortable at this point asking why it is allowed to put expectation inside the integral (which is not even an integral in Riemann-Stieltjes sense). However, this can be made rigorous by dealing with processes which have sample paths of bounded variation first and passing to the limit afterwards (cf. \cite{FV10G, FV10, FR11b, FH11}). We decided not to go too much into detail here in order not to distract the reader from the main ideas and to improve the readability.}
	By a $2D$ Young estimate (cf. \cite{T02}),
	\begin{align*}
		\left| \int_{[s,t]^2} R_{(X^i - Y^i)}\left( \begin{array}{c} s,u \\ s,v \end{array} \right) \, dR_{X^j}(u,v) \right|
		&\leq c_5(\rho,\gamma) V_{\gamma}(R_{(X^i - Y^i)};[s,t]^2) V_{\rho}(R_{X^j};[s,t]^2) \\
		&\leq c_6 V_{\gamma}(R_{(X - Y)};[s,t]^2) \omega(s,t)^{1/ \rho}.
	\end{align*}
%	(Strictly speaking, the first inequality is proven in \cite{T02} only for the case when $\frac{1}{\rho} + \frac{1}{\gamma} > 1$. However, the case $\rho =1, \gamma = \infty$ is much simpler and follows readily from the definition of a $2D$ Young integral, just as for the Riemann-Stieltjes integral in the one-dimensional case.) 
%	Now we consider the second term. The geometric structure assures that we have a integration by parts formula. Independence and the triangle inequality give
%	\begin{align*}
%		\left| \int_s^t Y^i_{s,u} \,d(X^j - Y^j)_u \right|_{L^2} \leq \left| \int_s^t (X^j - Y^j)_{s,u} \, dY^i_u \right|_{L^2} + |X_{s,t}^j - Y_{s,t}^j|_{L^2} |Y_{s,t}^i|_{L^2}.
%	\end{align*}
%	Now we can proceed as before to obtain the same estimate. 
	The second term is treated exactly in the same way.
	Summarizing, we have shown that
	\begin{align*}
		\left| \int_s^t X_{s,u} \,\otimes dX_u - \int_s^t Y_{s,u} \,\otimes dY_u\right|_{L^2} 
		\leq C \sqrt{V_{\gamma}(R_{(X - Y)};[s,t]^2)} \omega(s,t)^{\frac{1}{2 \rho}}.
	\end{align*}
	Finally, by Lemma \ref{lemma:infinterpolation}
	\begin{align*}
		V_{\gamma}(R_{(X - Y)};[s,t]^2) \leq V_{\infty}(R_{(X - Y)};[s,t]^2)^{1-\rho / \gamma} \omega(s,t)^{1/ \gamma}
	\end{align*}
	and by the Cauchy-Schwarz inequality
	\begin{align*}
		V_{\infty}(R_{(X - Y)};[s,t]^2) \leq 4 \sup_{u\in [s,t]} |X_u - Y_u|_{L^2}^2
	\end{align*}
	which gives the claim.
\end{proof}

\bigskip

\begin{cor} \label{first two levels deterministic estimates}
	Under the assumptions of Proposition \ref{prop_estimates_first_levels}, for every $\gamma$ satisfying $\gamma \geq \rho$ and $\frac{1}{\gamma} + \frac{1}{\rho} > 1$, and every $p >  2\rho$ and $\gamma' > \gamma$, 
	%for almost every enhanced sample path $(\X, \Y)$, 
	there is a (random) control $\hat{\omega}$ such that
	\begin{align}
		|\X_{s,t}^n| &\leq \hat{\omega}(s,t)^{n/p}\label{eqn_aspvarboundX}\\
		|\Y_{s,t}^n| &\leq \hat{\omega}(s,t)^{n/p}\label{eqn_aspvarboundY}\\
		|\X_{s,t}^n - \Y_{s,t}^n| &\leq \epsilon \hat{\omega}(s,t)^{\frac{1}{2\gamma'} + \frac{n-1}{p}}\label{eqn_aspvarboundX-Y}
	\end{align}
	holds a.s. for all $s<t$ and $n=1,2$ where $\epsilon =  \sup_{u\in[0,T]} \left| X_u - Y_u \right|_{L^2}^{1-\frac{\rho}{\gamma}}$. Furthermore, there is a constant $C=C(p,\rho,\gamma,\gamma',K)$ such that
	\begin{align*}
		\left|\hat{\omega}(0,T) \right|_{L^q} \leq C T (q^{p/2} + q^{\gamma'})
	\end{align*}
	holds for all $q\geq 1$.
\end{cor}

\begin{proof}
	Let $\omega$ be the control from Proposition \ref{prop_estimates_first_levels}. We know that
	\begin{align*}
		|\X^n_{s,t} - \Y^n_{s,t}|_{L^2} &\leq c_1 \epsilon \omega(s,t)^{\frac{1}{2\gamma} + \frac{n-1}{2\rho}}
	\end{align*}
	holds for a constant $c_1$ for all $s<t$ and $n=1,2$. Furthermore, $|\X^n_{s,t}|_{L^2} \leq c_2 \omega(s,t)^{\frac{n}{2\rho}}$ for a constant $c_2$ for all $s<t$ and $n=1,2$ and the same holds for $\Y$ (this just follows from setting $Y=\text{const.}$ and $\gamma = \rho$ in Proposition \ref{prop_estimates_first_levels}). Now introduce a new process $\tilde{\X}: [0,T] \rightarrow \R^{d}$ on the same sample space as $\X$ such that for all sample points, we have
	\begin{align*}
	\tilde{\X}_{\omega(0,t) / \omega(0,T)} = \X_{t}, \qquad \forall t \in [0,T], 
	\end{align*}
and define $\tilde{\Y}$ in the same way. Then $\tilde{\X}, \tilde{\Y}$ are well defined, multiplicative, and we can replace the control $\omega$ by $c_{3} K |t-s|$ for the two re-parametrized processes. Using that $X,Y$ are Gaussian, we may pass from $L^2$ to $L^q$ estimates and we know that $O(|\X^n|_{L^q}) = q^{n/2}$ (same for $\Y$ and $\X - \Y$, cf. \cite[Appendix A]{FV10}). Hence
	\begin{align}
		|\tilde{\X}^n_{s,t}|_{L^q} &\leq c_4 (\sqrt{q K^{1/\rho}})^n |t-s|^{\frac{n}{2\rho}}\label{eqn_lqhoelderboundX}\\
		|\tilde{\Y}^n_{s,t}|_{L^q} &\leq c_4 (\sqrt{q K^{1/\rho}})^n |t-s|^{\frac{n}{2\rho}}\label{eqn_lqhoelderboundY}\\
		|\tilde{\X}^n_{s,t} - \tilde{\Y}^n_{s,t}|_{L^q} &\leq \tilde{\epsilon} c_4 (\sqrt{q K^{1/\rho}})^n |t-s|^{\frac{1}{2\gamma} + \frac{(n-1)}{2\rho}}\label{eqn_lqhoelderboundX-Y}
	\end{align}
	hold for all $s<t$, $n=1,2$ and $q\geq 1$ with $\tilde{\epsilon} = \epsilon K^{\frac{1}{2\gamma} - \frac{1}{2\rho}}$. Using Lemma \ref{lemma_kolmogorov} in the appendix, we see that there is a constant $c_5 = c_5(p,\rho,\gamma,\gamma',K)$ such that
	\begin{align}
		\left| \sup_{s<t \in [0,T]} \frac{|\tilde{\X}^n_{s,t}|}{|t-s|^{n/p}} \right|_{L^q} &\leq c_5 q^{n/2}\\
		\left| \sup_{s<t \in [0,T]} \frac{|\tilde{\Y}^n_{s,t}|}{|t-s|^{n/p}} \right|_{L^q} &\leq c_5 q^{n/2}\\
		\left| \sup_{s<t \in [0,T]} \frac{|\tilde{\X}^n_{s,t} - \tilde{\Y}^n_{s,t}|}{|t-s|^{1/p(n)}} \right|_{L^q} &\leq \epsilon c_5 q^{n/2}
	\end{align}
	hold for $q$ sufficiently large and $n=1,2$ where $\frac{1}{p(n)} = \frac{1}{2\gamma'} + \frac{n-1}{p}$.
%	\begin{align*}
%		\frac{1}{p(n)} = \frac{1}{2\gamma'} + \frac{n-1}{p}.
%	\end{align*}
	Set
	\begin{align*}
		\hat{\omega}_X^n(s,t) &:= \sup_{D\subset [s,t]} \sum_{t_i \in D} |\X^n_{t_i,t_{i+1}}|^{p/n} \\
		\hat{\omega}_Y^n(s,t) &:= \sup_{D\subset [s,t]} \sum_{t_i \in D} |\Y^n_{t_i,t_{i+1}}|^{p/n} \\
		\hat{\omega}_{X-Y}^n(s,t) &:=  \sup_{D\subset [s,t]} \sum_{t_i \in D} |\X^n_{t_i,t_{i+1}} - \Y^n_{t_i,t_{i+1}}|^{p(n)}
	\end{align*}
	and
	\begin{align*}
		\hat{\omega}(s,t) := \sum_{n=1,2} \hat{\omega}_X^n(s,t) + \hat{\omega}_Y^n(s,t) + \epsilon^{\frac{1}{p(n)}} \hat{\omega}_{X-Y}^n(s,t).
	\end{align*}
	for $s<t$. Clearly, $\hat{\omega}$ fulfils \eqref{eqn_aspvarboundX}, \eqref{eqn_aspvarboundY} and \eqref{eqn_aspvarboundX-Y}. Moreover, the notion of $p$-variation is invariant under reparametrization, hence
	\begin{align*}
		\hat{\omega}_{X}^n(0,T) = \sup_{D\subset [0,T]} \sum_{t_i \in D} |\X^n_{t_i,t_{i+1}}|^{p/n} = \sup_{D\subset [0,T]} \sum_{t_i \in D} |\tilde{\X}^n_{t_i,t_{i+1}}|^{p/n} \leq T \sup_{s<t \in [0,T]} \frac{|\tilde{\X}^n_{s,t}|^{p/n}}{|t-s|}
	\end{align*}
	and a similar estimate holds for $\hat{\omega}_Y^n(0,T)$ and $\hat{\omega}_{X-Y}^n(0,T)$. By the triangle inequality and the estimates \eqref{eqn_lqhoelderboundX}, \eqref{eqn_lqhoelderboundY} and \eqref{eqn_lqhoelderboundX-Y},
	\begin{align*}
		\left| \hat{\omega}(0,T)\right|_{L^q} &\leq \sum_{n=1,2} |\hat{\omega}_X^n(0,T)|_{L^q} + |\hat{\omega}_Y^n(0,T)|_{L^q} + \epsilon^{\frac{1}{p(n)}} |\hat{\omega}_{X-Y}^n(0,T)|_{L^q}\\
		%&\leq T \sum_{n=1,2} \left| \sup_{s<t \in [0,T]} \frac{|\tilde{\X}^n_{s,t}|}{|t-s|^{n/p}} \right|_{L^{\frac{pq}{n}}}^{\frac{p}{n}} + \left| \sup_{s<t \in [0,T]} \frac{|\tilde{\Y}^n_{s,t}|}{|t-s|^{n/p}} \right|_{L^{\frac{pq}{n}}}^{\frac{p}{n}} + \epsilon^{-p(n)} \left| \sup_{s<t \in [0,T]}  \frac{|\tilde{\X}^n_{s,t}| - \tilde{\Y}^n_{s,t}}{|t-s|^{1/p(n)}} \right|_{L^{p(n)q}}^{p(n)}\\
		&\leq c_6 T \left( q^{p/2} + q^{\frac{p(1)}{2}} + q^{p(2)} \right) \leq c_7 T (q^{p/2} + q^{\gamma'})
	\end{align*}
	for $q$ large enough. We can extend the estimate to all $q\geq 1$ by making the constant larger if necessary.
	\end{proof}

\bigskip

\begin{cor}
Let $\hat{\omega}$ be the random control defined in the previous corollary. Then, for every $n$, there exists a constant $c_n$ such that
\begin{align*}
|\X_{s,t}^{n}| < c_{n} \hat{\omega}(s,t)^{\frac{n}{p}}, \qquad |\Y_{s,t}^{n}| < c_{n} \hat{\omega}(s,t)^{\frac{n}{p}}
\end{align*}
a.s. for all $s<t$. The constants $c_n$ are deterministic and can be chosen such that $c_{n} \leq \frac{2^{n}}{(n / p)!}$, where $x!:= \Gamma(x - 1)$.
\end{cor}

\begin{proof}
	Follows from the extension theorem, cf. \cite[Theorem 2.2.1]{L98} or \cite[Theorem 3.7]{LCL06}.
\end{proof}

\bigskip

\section{Main estimates}
In what follows, we let $p \in (2 \rho, 3)$. Let $\gamma \geq \rho$ such that $\frac{1}{\gamma} + \frac{1}{\rho} > 1$. We write $\log^{+}x = \max \{x, 0\}$, and set
\begin{align*}
\epsilon = \sup_{u \in [0, T]} |X_{u} - Y_{u}|_{L^{2}}^{1 - \frac{\rho}{\gamma}}
\end{align*}

\bigskip

\subsection{Higher level estimates}

\bigskip

We first introduce some notations. Suppose $\X$ is a multiplicative functional in $T^{N}(\R^{d})$ with finite $p$-variation controlled by $\omega$, $N \geq \lfl p \rfl$. Then, define
\begin{align*}
\hat{\X}_{s,t} = 1 + \sum_{n=1}^{N} \X_{s,t}^{n} \in T^{N+1}(\R^{d}). 
\end{align*}
Then, $\hat{\X}$ is multiplicative in $T^{N}$, but in general not in $T^{N+1}$. For any partition $D = \{s = u_{0} < u_{1} < \cdots < u_{L} < u_{L+1} = t\}$, define
\begin{align*}
\hat{\X}_{s,t}^{D}:= \hat{\X}_{s, u_{1}} \otimes \cdots \otimes \hat{\X}_{u_L, t} \in T^{N+1}(\R^{d}). 
\end{align*}
The following lemma gives a construction of the unique multiplicative extension of $\X$ to higher degrees. It was first proved in Theorem 2.2.1 in \cite{L98}.

\bigskip

\begin{lemma}\label{lemma_essentialexttheorem}
Let $\X$ be a multiplicative functional in $T^{N}$. Let $D = \{s < u_{1} < \cdots < u_L < t\}$ be any partition of $(s,t)$, and $D^{j}$ denote the partition with the point $u_{j}$ removed from $D$. Then, 
\begin{align} \label{removal of a point}
\hat{\X}_{s,t}^{D} - \hat{\X}_{s,t}^{D^{j}} = \sum_{n=1}^{N} \X_{u_{j-1}, u_{j}}^{n} \otimes \X_{u_{j}, u_{j+1}}^{N+1-n} \in T^{N+1}(\R^{d}). 
\end{align}
In particular, its projection onto the subspace $T^{N}$ is the $0$-vector. Suppose further that $\X$ has finite $p$-variation controlled by $\omega$, and $N \geq \lfl p \rfl$, then the limit
\begin{align*}
\lim_{|D| \rightarrow 0} \hat{\X}_{s,t}^{D} \in T^{N+1}(\R^{d})
\end{align*}
exists. Furthermore, it is the unique multiplicative extension of $\X$ to $T^{N+1}$ with finite $p$-variation controlled by $\omega$. 
\end{lemma}

\bigskip

\begin{theorem} \label{third level estimate}
Let $(X,Y)$ and $\rho,\gamma$ as in Proposition \ref{prop_estimates_first_levels}. Then for every $p>2\rho$ and $\gamma' > \gamma$ there exists a constant $C_{3}$ depending on $p$ and $\gamma'$ and a (random) control $\hat{\omega}$ such that for all $q \geq 1$, we have
\begin{align*}
|\hat{\omega}(0,T)|_{L^{q}} \leq M < +\infty, 
\end{align*}
where $M = M(p, \rho,\gamma,\gamma', K,q)$, and the following holds a.s. for all $[s,t]$:  
\begin{enumerate}
\item If $\frac{1}{2 \gamma'} + \frac{2}{p} > 1$, then
\begin{align*}
|\X_{s,t}^{3} - \Y_{s,t}^{3}| < C_{3} \epsilon \hat{\omega}(s,t)^{\frac{1}{2\gamma'} + \frac{2}{p}}. 
\end{align*}

\item If $\frac{1}{2 \gamma'} + \frac{2}{p} = 1$, then
\begin{align*}
|\X_{s,t}^{3} - \Y_{s,t}^{3}| < C_{3} \epsilon \cdot (1 + \log^{+} \big[\hat{\omega}(0,T) / \epsilon^{1 - \frac{p}{2 \gamma'}} \big] ) \cdot \hat{\omega}(s,t). 
\end{align*}

\item If $\frac{1}{2 \gamma'} + \frac{2}{p} < 1$, then
\begin{align*}
|\X_{s,t}^{3} - \Y_{s,t}^{3}| < C_{3} \epsilon^{\frac{3-p}{1 - p / 2\gamma'}} \hat{\omega}(s,t), 
\end{align*}
\end{enumerate}
\end{theorem}

\begin{proof}
Let $s < t \in [0,T]$ and let $\hat{\omega}$ be the (random) control defined in Corollary \ref{first two levels deterministic estimates}. Then, by the same corollary, for every $q \geq 1$, $|\hat{\omega}(0, T)|_{L^{q}} \leq M$. Fix an enhanced sample rough path $(\X, \Y)$ up to level $2$ and for simplicity, we will use $\omega$ to denote the corresponding realisation of the (random) control $\hat{\omega}$. We can assume without loss of generality that
\begin{align} \label{assumption on eta}
\epsilon < \omega(s,t)^{\frac{1}{p} - \frac{1}{2\gamma'}}, 
\end{align}
otherwise there will be nothing to prove. Let $D = \{s = u_{0} < \cdots < u_{L+1} = t\}$ be a dissection. Then (cf. \cite[Lemma 2.2.1]{L98}), there exists a $j$ such that
\begin{align} \label{the dropped point}
\omega(u_{j-1}, u_{j+1}) \leq \frac{2}{L}\omega(s,t), \qquad L \geq 1. 
\end{align}
Let $D^{j}$ denote the dissection with the point $u_{j}$ removed from $D$. Then, we have
\begin{align*}
|(\hat{\X}_{s,t}^{D} - \hat{\Y}_{s,t}^{D})^{3}| &< |(\hat{\X}_{s,t}^{D^{j}} - \hat{\Y}_{s,t}^{D^{j}})^{3}| + \sum_{k=1}^{2}(|\textbf{R}_{u_{j-1}, u_{j}}^{k} \otimes \textbf{X}_{u_{j}, u_{j+1}}^{3-k}| \\
&+ |\textbf{X}_{u_{j-1}, u_{j}}^{k} \otimes \textbf{R}_{u_{j}, u_{j+1}}^{3-k}| + |\textbf{R}_{u_{j-1}, u_{j}}^{k} \otimes \textbf{R}_{u_{j}, u_{j+1}}^{3-k}|), 
\end{align*}
where $\textbf{R}_{s,t} = \textbf{Y}_{s,t} - \textbf{X}_{s,t}$. By assumption,
\begin{align} \label{Wiener chaos estimate}
|\textbf{R}_{u_{j-1}, u_{j}}^{k} \otimes \textbf{R}_{u_{j}, u_{j+1}}^{3-k}| < C \cdot \min \bigg\{\epsilon \big(\frac{1}{L} \omega(s,t)\big)^{\frac{1}{2\gamma'} + \frac{2}{p}}, \big(\frac{1}{L}\omega(s,t)\big)^{\frac{3}{p}}  \bigg \}, 
\end{align}
and similar inequalities hold for the other two terms in the bracket. Thus, we have
\begin{align*}
|(\hat{\X}_{s,t}^{D} - \hat{\Y}_{s,t}^{D})^{3}| < |(\hat{\X}_{s,t}^{D^{j}} - \hat{\Y}_{s,t}^{D^{j}})^{3}| + C_{3} \min \bigg\{\epsilon \big(\frac{1}{L} \omega(s,t)\big)^{\frac{1}{2\gamma'} + \frac{2}{p}}, \big(\frac{1}{L}\omega(s,t)\big)^{\frac{3}{p}}  \bigg \}. 
\end{align*}
Let $N$ be the integer that
\begin{align} \label{the unique integer N}
[\frac{1}{N+1} \omega(s,t)]^{\frac{1}{p} - \frac{1}{2 \gamma'}} \leq \epsilon < [\frac{1}{N} \omega(s,t)]^{\frac{1}{p} - \frac{1}{2 \gamma'}}, 
\end{align}
then
\begin{align*}
\epsilon [\frac{1}{L} \omega(s,t)]^{\frac{1}{2\gamma'} + \frac{2}{p}} < [\frac{1}{L}\omega(s,t)]^{\frac{3}{p}}
\end{align*}
if and only if $L \leq N$. By Lemma \ref{lemma_essentialexttheorem}, we have
\begin{align*}
\X_{s,t}^{3} = \lim_{|D| \rightarrow 0} (\hat{\X}_{s,t}^{D})^{3}, \qquad \Y_{s,t}^{3} = \lim_{|D| \rightarrow 0} (\hat{\Y}_{s,t}^{D})^{3}. 
\end{align*}
Thus, for a fixed partition $D$, we choose a point each time according to \eqref{the dropped point}, and drop them successively. By letting $|D| \rightarrow +\infty$, we have
\begin{align*}
|\X_{s,t}^{3} - \Y_{s,t}^{3}| \leq C_{3} \bigg[ \epsilon \sum_{L=1}^{N} \big(\frac{1}{L}\omega(s,t)\big)^{\frac{1}{2\gamma'} + \frac{2}{p}} + \sum_{L=N+1}^{+\infty}\big(\frac{1}{L}\omega(s,t)\big)^{\frac{3}{p}} \bigg]. 
\end{align*}
Approximating the sums by integrals, we have
\begin{align*}
|\X_{s,t}^{3} - \Y_{s,t}^{3}| < C_{3} [\epsilon \omega(s,t)^{\frac{1}{2\gamma'} + \frac{2}{p}}(1 + \int_{1}^{N} x^{-(\frac{1}{2\gamma'} + \frac{2}{p})}dx) + \omega(s,t)^{\frac{3}{p}} \int_{N}^{+\infty} x^{- \frac{3}{p}}dx  ]. 
\end{align*}
Compute the second integral, and use
\begin{align*}
[\frac{1}{N+1} \omega(s,t)]^{(\frac{1}{p} - \frac{1}{2 \gamma'})}  \leq \epsilon, 
\end{align*}
we obtain
\begin{align} \label{inequality for all cases}
|\X_{s,t}^{3} - \Y_{s,t}^{3}| < C_{3} [\epsilon \omega(s,t)^{\frac{1}{2\gamma'} + \frac{2}{p}}(1 + \int_{1}^{N} x^{-(\frac{1}{2\gamma'} + \frac{2}{p})}dx) + \epsilon^{\frac{3-p}{1-p /2 \gamma'}} \omega(s,t) ]. 
\end{align}
Now we apply the above estimates to the three situations respectively.

\bigskip

\begin{flushleft}
\textbf{1.} $\frac{1}{2\gamma'} + \frac{2}{p} > 1$. 
\end{flushleft}
In this case, the integral
\begin{align*}
\int_{1}^{N} x^{-(\frac{1}{2\gamma'} + \frac{2}{p})}dx < \int_{1}^{+\infty} x^{-(\frac{1}{2\gamma'} + \frac{2}{p})}dx < +\infty
\end{align*}
converges. On the other hand, \eqref{assumption on eta} implies
\begin{align*}
\epsilon^{\frac{3-p}{1- p / 2 \gamma'}} \omega(s,t) < \epsilon \omega(s,t)^{\frac{1}{2\gamma'} + \frac{2}{p}}, 
\end{align*}
thus, from \eqref{inequality for all cases}, we get
\begin{align*}
|\X_{s,t}^{3} - \Y_{s,t}^{3}| < C_{3} \epsilon \omega(s,t)^{\frac{1}{2\gamma'} + \frac{2}{p}}. 
\end{align*}

\bigskip

\begin{flushleft}
\textbf{2.} $\frac{1}{2\gamma'} + \frac{2}{p} = 1$. 
\end{flushleft}
In this case, $\frac{1}{p} - \frac{1}{2 \gamma'} = \frac{3-p}{p}$, and $\frac{3 - p}{1 - p/2 \gamma'} = 1$. Thus, by the second inequality in \eqref{the unique integer N}, we have
\begin{align*}
\int_{1}^{N} x^{-1} dx = \log N < \log \omega(s,t) - \frac{p}{3 - p} \log \epsilon. 
\end{align*}
On the other hand, \eqref{assumption on eta} gives
\begin{align*}
\log \omega(s,t) - \frac{p}{3 - p} \log \epsilon > 0. 
\end{align*}
Combining the previous two bounds with \eqref{inequality for all cases}, we get
\begin{align*}
|\X_{s,t}^{3} - \Y_{s,t}^{3}| < C_{3} \epsilon [1 + \log \omega(s,t) - \frac{p}{3 - p} \log \epsilon] \omega(s,t). 
\end{align*}
We can simplify the above inequality to
\begin{align*}
|\X_{s,t}^{3} - \Y_{s,t}^{3}| < C_{3} \epsilon[1 + \log^{+}(\omega(0,T) / \epsilon^{\frac{p}{3-p}})] \omega(s,t), 
\end{align*}
where we have also included the possibility of $\epsilon \geq \omega(0,T)^{\frac{3}{p} - 1}$.

\bigskip

\begin{flushleft}
\textbf{3.} $\frac{1}{2\gamma'} + \frac{2}{p} < 1$. 
\end{flushleft}
Now we have
\begin{align*}
1 + \int_{1}^{N} x^{- (\frac{1}{2\gamma'} + \frac{2}{p})} dx < C N^{1 - \frac{1}{2\gamma'} - \frac{2}{p}} < C \cdot \epsilon^{-(1 - \frac{1}{2\gamma'} - \frac{2}{p}) / \frac{1}{p} - \frac{1}{2 \gamma'}} \omega(s,t)^{1 - \frac{1}{2\gamma'} - \frac{2}{p}}, 
\end{align*}
where the second inequality follows from \eqref{the unique integer N}. Combining the above bound with \eqref{inequality for all cases}, we obtain
\begin{align*}
|\X_{s,t}^{3} - \Y_{s,t}^{3}| < C_{3} \epsilon^{\frac{3 - p}{1 - p / 2 \gamma'}} \omega(s,t). 
\end{align*}

\end{proof}

\bigskip

The following theorem, obtained with the standard induction argument, gives estimates for all levels $n = 1, 2, \cdots$.

\bigskip

\begin{theorem} \label{main estimate}
Let $(X,Y)$ and $\rho,\gamma$ as in Proposition \ref{prop_estimates_first_levels}, $p>2\rho$ and $\gamma' > \gamma$. Then there exists a (random) control $\hat{\omega}$ such that for every $q \geq 1$, we have
\begin{align*}
|\hat{\omega}(0,T)|_{L^{q}} \leq M
\end{align*}
where $M = M(p, \rho, \gamma, \gamma',q,K)$, and for each $n$ there exists a (deterministic) constant $C_{n}$ depending on $p$ and $\gamma'$ such that a.s. for all $[s,t]$:  
\begin{enumerate}
\item If $\frac{1}{2\gamma'} + \frac{2}{p} > 1$, then we have
\begin{align*}
|\X_{s,t}^{n} - \Y_{s,t}^{n}| < C_{n} \epsilon \hat{\omega}(s,t)^{\frac{1}{2\gamma'} + \frac{n-1}{p}}
\end{align*}

\item If $\frac{1}{2\gamma'} + \frac{2}{p} = 1$, then we have
\begin{align*}
|\X_{s,t}^{n} - \Y_{s,t}^{n}| < C_{n} \epsilon \cdot (1 + \log^{+} \big[\hat{\omega}(0,T) / \epsilon^{1 - \frac{p}{2 \gamma'}} \big]) \cdot \hat{\omega}(s,t)^{\frac{1}{2 \gamma'} + \frac{n-1}{p}}. 
\end{align*}

\item If $\frac{1}{2\gamma'} + \frac{2}{p} < 1$, then for all $s < t$ and all small $\epsilon$, we have
\begin{align} \label{the third inequality}
	|\X_{s,t}^{n} - \Y_{s,t}^{n}| < C_{n} \epsilon^{\frac{3-p}{1 - p / 2 \gamma'}} \hat{\omega}(s,t)^{\frac{n-1+\{p\}}{p}}. 
\end{align}

\end{enumerate}
\end{theorem}

\begin{proof}
We prove the case when $\frac{1}{2\gamma'} + \frac{2}{p} < 1$; the other two situations are similar. Let $\hat{\omega}$ be the control in the previous theorem. Fix an enhanced sample path $(\X, \Y)$, the corresponding realisation $\omega$ of $\hat{\omega}$, and $s < t \in [0, T]$. We may still assume \eqref{assumption on eta} without loss of generality. Thus, for $n = 1, 2$, we have
\begin{align*}
|\X_{s,t}^{n} - \Y_{s,t}^{n}| < \epsilon \omega(s,t)^{\frac{1}{2 \gamma'} + \frac{n-1}{p}} < C_{n} \epsilon^{\frac{3-p}{1 - p /2 \gamma'}} \omega(s,t)^{\frac{n-1+\{p\}}{p}},
\end{align*}
where the second inequality comes from \eqref{assumption on eta}. The above inequality also holds for $k = 3$ by the previous theorem. Now, suppose \eqref{the third inequality} holds for $k = 1, \cdots, n$, where $n \geq 3$, then for level $k = n + 1$, the exponent is expected to be
\begin{align*}
\frac{n + \{p\}}{p} > 1, 
\end{align*}
so that the usual induction procedure works (cf. \cite{L98}, Theorem 2.2.2.). Thus, we prove \eqref{the third inequality} for all $n$. 
\end{proof}

\bigskip

\subsection{Proof of Theorem \ref{main theorem}}

\begin{proof}
We prove the second situation when $\frac{1}{2\gamma} + \frac{1}{\rho} \leq 1$. The first one is similar. Let $\epsilon = \sup_{u \in [0, T]} |X_{u} - Y_{u}|_{L^{2}}^{1 - \frac{\rho}{\gamma}}$. It is sufficient to show that for every $p>2\rho$ there is a constant $C$ such that
\begin{align*}
	|\varrho_{\sigma-\text{var}}^{N}(\X, \Y)|_{L^{q}} \leq C \epsilon^{\frac{3 - p}{1 - \rho / \gamma}}, 
\end{align*}
where $\sigma > 2 \gamma$ and $N \geq \lfl \sigma \rfl$ both satisfy the assumptions of Theorem \ref{main theorem}. Set
\begin{align*}
	\rho':=(1+\eta)\rho,\quad p := 2(1+2\eta)\rho, \quad \gamma' := (1+\eta)\gamma, \quad \gamma'' := (1 + 2\eta)\gamma
\end{align*}
for some $\eta>0$. We can choose $\eta$ small enough such that $\frac{1}{\rho'} + \frac{1}{\gamma'}>1$ and $p<3$ hold, and the conditions of Theorem \ref{main estimate} are satisfied for $\rho'$ and $\gamma'$. Clearly $\frac{1}{\gamma''} + \frac{2}{p} < \frac{1}{2\gamma} + \frac{2}{p} \leq 1 $, thus Theorem \ref{main estimate} implies that
\begin{align*}
	|\X_{s,t}^{n} - \Y_{s,t}^{n}| < C_{n} \epsilon^{\frac{3-p}{1 - \rho / \gamma}} \hat{\omega}(s,t)^{\frac{n-1+\{p\}}{p}}
\end{align*} 
holds a.s. for any $n$ and $s<t$ where $\hat{\omega}$ is a random control as in Theorem \ref{main estimate}. Furthermore, for any $n$,
\begin{align} \label{rewrite the exponents}
\frac{n - 1 + \{p\}}{p} = \frac{n}{2 \gamma''} + n(\frac{1}{p} - \frac{1}{2\gamma''}) - \frac{1 - \{p\}}{p} = \frac{n}{2 \gamma''} + (n-1)(\frac{1}{p} - \frac{1}{2 \gamma''}) + (1 - \frac{2}{p} - \frac{1}{2\gamma''}). 
\end{align}
Note that the last expression implies that
\begin{align*}
	\theta_n := n(\frac{1}{p} - \frac{1}{2\gamma''}) - \frac{1 - \{p\}}{p} >0 
\end{align*}
for all $n$. Fix a dissection $D = \{0 = u_0 < \ldots < u_L < T\}$ of the interval $[0,T]$. Using $\hat{\omega}(u_{i}, u_{i+1}) \leq \hat{\omega}(0,T)$, we have
\begin{align*}
\bigg( \sum_{i}|\X_{u_{i}, u_{i+1}}^{n} - \Y_{u_{i}, u_{i+1}}^{n}|^{\frac{\sigma}{n}} \bigg)^{\frac{n}{\sigma}} \leq C_{n} \epsilon^{\frac{3 - p}{1 - \rho / \gamma}} \hat{\omega}(0,T)^{\theta_n}  \bigg(\sum_{i} \hat{\omega}(u_{i}, u_{i+1})^{\frac{\sigma}{2\gamma''}} \bigg)^{\frac{n}{\sigma}}.
\end{align*}
Choosing $\eta$ smaller if necessary, we may assume that $\sigma \geq 2\gamma''$ and super-additivity of the control implies
\begin{align*}
\bigg(\sum_{i} \hat{\omega}(u_{i}, u_{i+1})^{\frac{\sigma}{2\gamma''}} \bigg)^{\frac{n}{\sigma}} \leq \hat{\omega}(0,T)^{\frac{n}{2 \gamma''}}. 
\end{align*}
Passing to the supremum over all partitions of $[0,T]$, we have
\begin{align*}
\sup_{D} \bigg( \sum_{i}|\X_{u_{i}, u_{i+1}}^{n} - \Y_{u_{i}, u_{i+1}}^{n}|^{\frac{\sigma}{n}} \bigg)^{\frac{n}{\sigma}} \leq C_{n} \epsilon^{\frac{3 - p}{1 - \rho / \gamma}} \hat{\omega}(0,T)^{\frac{n-1+\{p\}}{p}}. 
\end{align*}
Let $q \geq 1$. By Theorem \ref{main estimate}, there is a constant $M$ depending on $\rho, \gamma, \sigma, \delta, q$ and $K$ such that $|\hat{\omega}(0,T)|_{L^{q}} \leq M$. Taking $L^{q}$ norm on both sides, we have
\begin{align*}
|\varrho_{\sigma-\text{var}}^{N}(\X, \Y)|_{L^{q}} \leq C \epsilon^{\frac{3 - p}{1 - \rho / \gamma}}
\end{align*}
which was the claim.
\end{proof}

\bigskip

\bigskip

\section{Applications}

\subsection{Convergence rates of rough differential equation}

\bigskip

Consider the \textit{rough differential equation} of the form
\begin{align} \label{rough differential equation}
dY_{t} = \sum_{i=1}^d V_i(Y_t)\, dX^i_t =: V(Y_{t})\, dX_{t};\quad Y_0 \in \R^e
\end{align}
where $X$ is a centered Gaussian process in $\R^{d}$ with independent components and $V = (V_i)_{i=1}^d$ a collection of bounded, smooth vector fields with bounded derivatives in $\R^e$. Rough path theory gives meaning to the pathwise solution to \eqref{rough differential equation} in the case when the covariance $R_X$ has finite $\rho$-variation for some $\rho<2$. Assume that $\rho \in [1, \frac{3}{2})$ and that there is a constant $K$ such that
\begin{align}\label{eqn_hoelder_cond_rho_var}
	V_{\rho}(R_X;[s,t]^2) \leq K |t-s|^{\frac{1}{\rho}}
\end{align}
for all $s<t$ (note that this condition implies that the sample paths of $X$ are $\alpha$-H\"older for all $\alpha<\frac{1}{2\rho}$). For simplicity, we also assume that $[0,T]=[0,1]$. For every $k\in\N$, we can approximate the sample paths of $X$ piecewise linear at the time points $\{0 < 1/k < 2/k < \ldots < (k-1)/k < 1\}$. We will denote this process by $X^{(k)}$. Clearly, $X^{(k)} \to X$ uniformly as $k\to\infty$. Now we substitute $X$ by $X^{(k)}$ in \eqref{rough differential equation}, solve the equation and obtain a solution $Y^{(k)}$; we call this the Wong-Zakai approximation of $Y$. One can show, using rough path theory, that $Y^{(k)} \to Y$ a.s. in uniform topology as $k\to\infty$. The proposition below is an immediate consequence of Theorem \ref{main theorem} and gives us rates of convergence.

\bigskip

\begin{prop}
The mesh size $\frac{1}{k}$ Wong-Zakai approximation converges uniformly to the solution of \eqref{rough differential equation} with a.s. rate at least $k^{ - (\frac{3}{2\rho} - 1 - \delta)}$ for any $\delta \in (0, \frac{3}{2\rho} - 1)$. In particular, the rate is arbitrarily close to $\frac{1}{2}$ when $\rho = 1$, which is the sharp rate in that case.
\end{prop}

\begin{proof}
First, one shows that \eqref{eqn_hoelder_cond_rho_var} implies that
\begin{align*}
\sup_{t \in [0,1]} |X_{t}^{(k)} - X_{t}|_{L^{2}} = O(k^{- \frac{1}{2\rho}}). 
\end{align*}
One can show (cf. \cite[Chapter 15.2.3]{FV10}) that there is a constant $C$ such that
\begin{align*}
	\sup_{k\in\N} V_{\rho}(R_{(X,X^{(k)})};[s,t]^2) \leq C |t-s|^{\frac{1}{\rho}}
\end{align*}
holds for all $s<t$. By choosing $q$ large enough, a Borel-Cantelli type argument applied to 
Theorem \ref{main theorem} shows that $\varrho_{\sigma-\text{var}}(\X,\X^{(k)}) \to 0$ a.s. for $k\to\infty$ with rate arbitrarily close to
\begin{align*}
k^{ - \frac{1}{2} (\frac{1}{\rho} - \frac{1}{\gamma})}, \qquad \text{if} \phantom{1} \frac{1}{2\gamma} + \frac{1}{\rho} > 1, 
\end{align*}
and arbitrarily close to
\begin{align*}
k^{- (3 - 2\rho)}, \qquad \text{if} \phantom{1} \frac{1}{2\gamma} + \frac{1}{\rho} \leq 1, 
\end{align*}
both cases are subject to $\gamma \geq \frac{3}{2}$ and $\frac{1}{\gamma} + \frac{1}{\rho} > 1$. Note that in the second situation, the actual value of $\gamma$ does not matter, and we always have a rate of 'almost' $\frac{3}{2\rho} - 1$. For the first situation, we need to let $\gamma$ as large as possible but still satisfy the constraints. The critical value is $\frac{1}{\gamma^{*}} = \frac{1}{2}(1 - \frac{1}{\rho})$, which also results in a rate that is arbitrarily close to $\frac{3}{2\rho} - 1$. Using the local Lipschitz property of the It\=o Lyons map (cf. \cite[Theorem 10.26]{FV10}), we conclude that the Wong-Zakai convergence rate is faster than
\begin{align*}
	k^{ - (\frac{3}{2\rho} - 1 - \delta)}
\end{align*}
for any $\delta > 0$ (but not for $\delta = 0$).
\end{proof}

\bigskip

\begin{remark}
For $\rho \in (1, \frac{3}{2})$, the rate above is not optimal. In fact, the sharp rate in this case is 'almost' $\frac{1}{\rho} - \frac{1}{2}$, as shown in \cite{FR11b}. The reason for the non-optimality of the rate is that we obtain the third level estimate merely based on the first two levels, which leads to a reduction in the exponent in the rate. On the other hand, this method does not use any Gaussian structure on the third level, and can be applied to more general processes. For the case $\rho = 1$, we recover the sharp rate of 'almost' $\frac{1}{2}$. 
\end{remark}

\bigskip

\subsection{The stochastic heat equation}

\bigskip

In the theory of stochastic partial differential equations (SPDEs), one typically considers the SPDE as an evolution equation in a function space. When it comes to the question of time and space regularity of the solution, one discovers that they will depend on the particular choice of this space. As a rule of thumb, the smaller the space, the lower the time regularity (\cite{H09}, Section 5.1). The most prominent examples of such spaces are Hilbert spaces, typically Sobolev spaces. However, in some cases, it can be useful to choose rough paths spaces instead (\cite{H11}). A natural question now is whether the known regularity results for Hilbert spaces are also true for rough paths spaces. In this section, we study the example of a modified stochastic heat equation for which we can give a positive answer.

Consider the stochastic heat equation: 
\begin{align}\label{eqn_modif_heat_eq}
	d\psi = (\partial_{x x} - 1)\psi\,dt + \sigma\,dW
\end{align}
where $\sigma$ is a positive constant, the spatial variable $x$ takes values in $[0,2\pi]$, $W$ is space-time white noise, i.e. a standard cylindrical Wiener process on $L^2([0,2\pi],\R^d)$, and $\psi$ denotes the stationary solution with values in $\R^d$. The solution $\psi$ is expected to be almost $\frac{1}{4}$-H\"{o}lder continuous in time and almost $\frac{1}{2}$-H\"{o}lder continuous in space (cf. \cite{H09}). In the next Theorem, we show that this is indeed the case if we choose the appropriate rough paths space.

\bigskip

\begin{theorem}
	Let $p > 2$. Then, for any fixed $t \geq 0$, the process $x\mapsto \psi_t(x)$ is a Gaussian process (in space) which can be lifted to an enhanced Gaussian process $\Psi_t(\cdot)$, a process with sample paths in $C^{0,p-\text{var}}([0,2\pi],G^{\lfloor p \rfloor}(\R^d))$. Moreover, $t \mapsto \Psi_{t}(\cdot)$ has a H\"{o}lder continuous modification 
	%in the space $C^{0,p-\text{var}}([0,2\pi],G^{\lfloor \sigma \rfloor}(\R^d))$ 
	(which we denote by the same symbol). More precisely, for every $\alpha \in \left(0, \frac{1}{4} - \frac{1}{2p}\right)$, there exists a (random) constant $C$ such that
	%for every $\alpha \in (0, \frac{1}{4})$ and every $s, t \in [0, T]$ and $p$ large enough, we have
\begin{align*}
	\varrho_{p-\text{var}}(\Psi_s,\Psi_t) \leq C |t - s|^{\alpha}
\end{align*}
	holds almost surely for all $s<t$. In particular, choosing $p$ large gives a time regularity of almost $\frac{1}{4}$-H\"{o}lder.
\end{theorem}

\begin{proof}
	The fact that $x \mapsto \psi_t(x)$ can be lifted to a process with rough sample paths and that there is some H\"{o}lder-continuity in time was shown in Lemma 3.1 in \cite{H11}, see also \cite{FH11}. We quickly repeat the argument and show where we can use our results in order to derive the exact H\"{o}lder exponents. Using the standard Fourier basis
	\begin{align*}
		e_k(x) = \begin{cases}
			\frac{1}{\sqrt{\pi}} \sin(kx) &\text{if } k > 0\\
			\frac{1}{\sqrt{2\pi}} &\text{if } k = 0\\
			\frac{1}{\sqrt{\pi}} \cos(kx) &\text{if } k < 0
			\end{cases}
	\end{align*}
	the equation \eqref{eqn_modif_heat_eq} can be rewritten as a system of SDEs
	\begin{align*}
		dY^k_t = -(k^2 + 1)Y^k_t\, dt + \sigma \, dW^k_t
	\end{align*}
	where $(W^k)_{k\in\Z}$ is a collection of independent standard Brownian motions and $(Y^k)_{k\in\Z}$ are the stationary solutions of the SDEs, i.e. a collection of centered, independent, stationary Ornstein-Uhlenbeck processes. The solution of \eqref{eqn_modif_heat_eq} is thus given by the infinite sum $\psi_t(x) = \sum_{k\in\Z} Y^k_t e_k(x)$. One can easily see that
	\begin{align*}
		E\left[ \psi_s(x) \otimes \psi_t(y) \right] = \frac{\sigma^2}{4\pi} \sum_{k\in \Z} \frac{\cos(k(x-y))}{1+k^2} e^{-(1+k^2)|t-s|} \times I_d
	\end{align*}
	where $I_d$ denotes the identity matrix in $\R^{d\times d}$. In particular, for $s=t$,
	\begin{align*}
		E\left[ \psi_t(x) \otimes \psi_t(y) \right] = K(x-y) \times I_d
	\end{align*}
	where $K$ is given by
	\begin{align*}
		K(x) = \frac{\sigma^2}{4\sinh(\pi)} \cosh(|x|-\pi)
	\end{align*}
	for $x\in[-\pi,\pi]$ and extended periodically for the remaining values of $x$ (this can be derived by a Fourier expansion of the function $x\mapsto \cosh(|x|-\pi)$). In particular, one can calculate that $x\mapsto \psi_t(x)$ is a Gaussian process with covariance of finite $1$-variation (see the remark at the end of the section for this fact), hence $\psi_t$ can be lifted to process $\Psi_t$ with sample paths in the rough paths space $C^{0,p-\text{var}}([0,2\pi],G^{\lfloor p \rfloor}(\R^d))$ for any $p > 2$. 
	
Furthermore, for any $s<t$, $x \mapsto (\psi_s(x),\psi_t(x))$ is a Gaussian process which fulfils the assumptions of Theorem \ref{main theorem} and the covariance $R_{(\psi_s,\psi_t)}$ also has finite $1$-variation, uniformly bounded for all $s<t$, hence
	\begin{align*}
		\sup_{s<t} |R_{(\psi_s,\psi_t)}|_{1-\text{var};[0,2\pi]^2} =: c_1 <\infty.
	\end{align*}
	Therefore, for any $\gamma \in (1,p/2)$ and $q \geq 1$ there is a constant $C=C(p,\gamma,c_1,q)$ such that
	\begin{align*}
		\left| \varrho_{p-\text{var}}(\Psi_s,\Psi_t) \right|_{L^q} \leq C \sup_{x\in [0,2\pi]}|\psi_t(x) - \psi_s(x)|_{L^2}^{1-\frac{1}{\gamma}}
	\end{align*}
	holds for all $s<t$. A straightforward calculation (cf. \cite[Lemma 3.1]{H11}) shows that
	\begin{align*}
		|\psi_t(x) - \psi_s(x)|_{L^2} \leq c_2 |t-s|^{1/4}
	\end{align*}
	for a constant $c_2$. In particular, we can find $\gamma$ and $q$ large enough such that
	\begin{align*}
  \alpha < \frac{\frac{q}{4}(1 - \frac{1}{\gamma}) - 1}{q} = \left(\frac{1}{4} - \frac{1}{4\gamma}\right) - \frac{1}{q} < \frac{1}{4} - \frac{1}{2p}. 
	\end{align*}
	%In turn, $p$ needs to be large enough such that $p \geq 2\gamma$. 
	Since $C^{0,p-\text{var}}$ is a Polish space, we can apply the usual Kolmogorov continuity criterion to conclude. 
	%that $t \mapsto \Psi_t$ has indeed $\alpha$-H\"{o}lder continuous sample paths.
\end{proof}

\bigskip

\begin{remark}
We emphasize that here, for every fixed $t$, the process $\psi_{t}(\cdot)$ is a Gaussian process, where the spatial variable $x$ should now be viewed as 'time'. This idea is due to M.Hairer. Knowing that the spatial regularity is 'almost' $1/2$ for every fixed time $t$, one could guess that covariance of this spatial Gaussian process has finite $1$-variation. For a formal calculation, we refer to \cite{H09} or \cite{FH11}.
\end{remark}

\bigskip

\bigskip

\section{Appendix}

The next Lemma is a slight modification of \cite[Theorem A.13]{FV10}. The proof follows the ideas of \cite[Theorem 3.1]{FH11}.

\begin{lemma}[Kolmogorov for multiplicative functionals]\label{lemma_kolmogorov}
	Let $\X, \Y\colon [0,T]\times\Omega \to T^N(V)$ be random multiplicative functionals and assume that $\X(\omega)$ and $\Y(\omega)$ are continuous for all $\omega\in\Omega$. Let $\beta, \delta \in (0,1]$ and choose $\beta' < \beta$ and $\delta' < \delta$. Assume that there is a constant $M>0$ such that
	\begin{align*}
		|\X^n_{s,t}|_{L^{q/n}} &\leq M^n |t-s|^{n\beta} \\
		|\Y^n_{s,t}|_{L^{q/n}} &\leq M^n |t-s|^{n\beta} \\
		|\X^n_{s,t} - \Y^n_{s,t}|_{L^{q/n}} &\leq M^n \epsilon |t-s|^{\delta + (n-1)\beta}
	\end{align*}
	hold for all $s<t \in [0,T]$ and $n=1,\ldots,N$ where $\epsilon$ is a positive constant and $q \geq q_0$ where
	\begin{align*}
		q_0 := 1 + \left( \frac{1}{\beta - \beta'} \vee \frac{1}{\delta - \delta'} \right). 
	\end{align*}
	Then there is a constant $C=C(N,\beta,\beta',\delta,\delta')$ such that
	\begin{align}
		\left| \sup_{s<t \in [0,T]} \frac{|\X^n_{s,t}|}{|t-s|^{n\beta'}} \right|_{L^{\frac{q}{n}}} &\leq C M^n \label{eqn_hoelderforX}\\
		\left| \sup_{s<t \in [0,T]} \frac{|\Y^n_{s,t}|}{|t-s|^{n\beta'}} \right|_{L^{\frac{q}{n}}} &\leq C M^n \label{eqn_hoelderforY}\\
		\left| \sup_{s<t \in [0,T]} \frac{|\X^n_{s,t} - \Y^n_{s,t}|}{|t-s|^{\delta' + (n-1)\beta'}} \right|_{L^{\frac{q}{n}}} &\leq C M^n \epsilon \label{eqn_hoelderforXY}
	\end{align}
	hold for all $n=1,\ldots,N$.
\end{lemma}

\begin{proof}
	W.l.o.g., we may assume $T=1$. Let $(D_k)_{k\in\N}$ be the sequence of dyadic partitions of the interval $[0,1)$, i.e. $D_k = \left\{ \frac{l}{2^k}\ :\ l=0,\ldots,2^k-1 \right\}$. Clearly, $|D_k| = \frac{1}{\# D_k} = 2^{-k}$. Set
	\begin{align*}
		K_{k,X}^n &:= \max_{t_i \in D_k} |\X_{t_i,t_{i+1}}^n| \\
		K_{k,Y}^n &:= \max_{t_i \in D_k} |\Y_{t_i,t_{i+1}}^n| \\
		K_{k,X-Y}^n &:= \frac{1}{\epsilon} \max_{t_i \in D_k} |\X_{t_i,t_{i+1}}^n - \Y_{t_i,t_{i+1}}^n| \\
	\end{align*}
	for $n=1,\ldots,N$ and $k\in\N$. By assumption, we have
	\begin{align*}
		E |K_{k,X}^n|^{\frac{q}{n}} \leq E\sum_{t_i \in D_k} |\X^n_{t_i,t_{i+1}}|^{\frac{q}{n}} \leq \#D_k \max_{t_i \in D_k} E|\X^n_{t_i,t_{i+1}}|^{\frac{q}{n}} \leq M^q |D_k|^{q\beta -1}.
	\end{align*}
	In the same way one estimates $K_{k,Y}^n$ and $K_{k,X - Y}^n$, hence
	\begin{align}
		|K_{k,X}^n|_{L^{q/n}} &\leq M^n |D_k|^{n\beta - n/q} \label{eqn_estim_KkX}\\
		|K_{k,Y}^n|_{L^{q/n}} &\leq M^n |D_k|^{n\beta - n/q} \label{eqn_estim_KkY}\\
		|K_{k,X - Y}^n|_{L^{q/n}} &\leq M^n |D_k|^{\delta + (n-1)\beta - n/q}\label{eqn_estim_KkXY}.
	\end{align}
	Note the following fact: For any dyadic rationals $s<t$, i.e. $s<t \in \Delta := \bigcup_{k=1}^{\infty} D_k$, there is a $m\in\N$ such that $|D_{m+1}| < |t-s| \leq |D_m|$ and a partition
	\begin{align}\label{eqn_smart_partition}
		s=\tau_0 < \tau_1 < \ldots < \tau_N = t
	\end{align}
	of the interval $[s,t)$ with the property that for any $i=0,\ldots,N-1$ there is a $k\geq m+1$ with $[\tau_i,\tau_{i+1}) \in D_k$, but for fixed $k\geq m+1$ there are at most two such intervals contained in $D_k$.\\
	
	\textbf{Step 1:} We claim that for every $n=1,\ldots,N$ there is a real random variable $K_X^n$ such that $|K_X^n|_{L^{q/n}} \leq M^n c$ where $c = c(\beta,\beta',\delta,\delta')$ and that for any dyadic rationals $s<t$ and $m$, $(\tau_i)_{i=0}^N$ chosen as in \eqref{eqn_smart_partition} we have
	\begin{align}\label{eqn_step1}
		\sum_{i=0}^{N-1} \frac{|\X^n_{\tau_i,\tau_{i+1}}|}{|t-s|^{n\beta'}} \leq K_X^n.
	\end{align}
	Furthermore, the estimate \eqref{eqn_step1} also holds for $\Y^n$ and a random variable $K_Y^n$. Indeed: By the choice of $m$ and $(\tau_i)_{i=0}^N$,
	\begin{align*}
		\sum_{i=0}^{N-1} \frac{|\X^n_{\tau_i,\tau_{i+1}}|}{|t-s|^{n\beta'}} \leq \sum_{k=m+1}^{\infty} \frac{2 K_{k,X}^n}{|D_{m+1}|^{n \beta'}} \leq 2 \sum_{k=m+1}^{\infty} \frac{K_{k,X}^n}{|D_k|^{n \beta'}} \leq 2 \sum_{k=1}^{\infty} \frac{K_{k,X}^n}{|D_k|^{n \beta'}} =: K_X^n.
	\end{align*}
	It remains to prove that $|K_X^n|_{L^{q/n}} \leq M^n c$. By the triangle inequality and the estimate \eqref{eqn_estim_KkX},
	\begin{align*}
		\left| \sum_{k=1}^{\infty} \frac{K_{k,X}^n}{|D_k|^{n \beta'}} \right|_{L^{q/n}} \leq M^n \sum_{k=1}^{\infty} |D_k|^{n(\beta - 1/q - \beta')} \leq M^n \sum_{k=1}^{\infty} |D_k|^{(\beta - 1/q_0 - \beta')} < \infty
	\end{align*}
	since $\beta - 1/q_0 - \beta' > 0$ which shows the claim. \\
	
	\textbf{Step 2:} We show that \eqref{eqn_hoelderforX} and \eqref{eqn_hoelderforY} hold for all $n=1,\ldots,N$. It is enough to consider $\X$. Note first that, due to continuity, it is enough to show the estimate for $\sup_{s<t \in \Delta} \frac{|\X^n_{s,t}|}{|t-s|^{n\beta'}}$. By induction over $n$: For $n=1$, this just follows from the usual Kolmogorov continuity criterion. Assume that the estimate is proven up to level $n-1$. Let $s<t$ be any dyadic rationals and choose $m$ and $(\tau_i)_{i=0}^N$ as in \eqref{eqn_smart_partition}. Since $\X$ is a multiplicative functional,
	\begin{align*}
		|\X^n_{s,t}| \leq \sum_{i=0}^{N-1} |\X^n_{\tau_i,\tau_{i+1}}| + \sum_{l=1}^{n-1} \max_{i=1,\ldots,N} |\X_{s,\tau_i}^{n-l}| \sum_{i=0}^{N-1} |\X^l_{\tau_i,\tau_{i+1}}|
	\end{align*}
	and thus, using step 1,
	\begin{align*}
		\frac{|\X^n_{s,t}|}{|t-s|^{n\beta'}} &\leq \sum_{i=0}^{N-1} \frac{|\X^n_{\tau_i,\tau_{i+1}}|}{|t-s|^{n\beta'}} + \sum_{l=1}^{n-1} \sup_{u<v \in \Delta} \frac{|\X^{n-l}_{u,v}|}{|v-u|^{(n-l)\beta'}} \sum_{i=0}^{N-1}
		\frac{|\X^l_{\tau_i,\tau_{i+1}}|}{|t-s|^{l\beta'}}\\
		&\leq K_X^n + \sum_{l=1}^{n-1} \sup_{u<v \in \Delta} \frac{|\X^{n-l}_{u,v}|}{|v-u|^{(n-l)\beta'}} K_X^l.
	\end{align*}
	We can now take the supremum over all $s<t \in \Delta$ on the left. Taking the $L^{q/n}$-norm on both sides, using first the triangle, then the H\"{o}lder inequality and the estimates from step 1 together with the induction hypothesis gives the claim.\\
	
	\textbf{Step 3:} As in step 1, we claim that for any $n=1,\ldots,N$ there is a random variable $K_{X-Y}^n \in L^{q/n}$ such that for any dyadic rationals $s<t$ and $m$, $(\tau_i)_{i=0}^N$ chosen as above we have
	\begin{align}\label{eqn_step3}
		\sum_{i=0}^{N-1} \frac{|\X^n_{\tau_i,\tau_{i+1}} - \Y^n_{\tau_i,\tau_{i+1}}|}{|t-s|^{\delta' + (n-1)\beta'}} \leq K_{X-Y}^n \epsilon.
	\end{align}
	Furthermore, we claim that $|K_{X-Y}^n|_{L^{q/n}} \leq M^n \tilde{c}$ where $\tilde{c}=\tilde{c}(\beta,\beta',\delta,\delta')$. The proof follows the lines of step 1, setting
	\begin{align*}
		\frac{1}{\epsilon}\sum_{i=0}^{N-1} \frac{|\X^n_{\tau_i,\tau_{i+1}} - \Y^n_{\tau_i,\tau_{i+1}} |}{|t-s|^{\delta' + (n-1)\beta'}} \leq 2 \sum_{k=1}^{\infty} \frac{K_{k,X-Y}^n}{|D_k|^{\delta' + (n-1) \beta'}} =: K_{X-Y}^n.
	\end{align*}

	\textbf{Step 4:} We prove that \eqref{eqn_hoelderforXY} holds for all $n=1,\ldots,N$. By induction over $n$: The case $n=1$ is again just the usual Kolmogorov continuity criterion applied to $t\mapsto \epsilon^{-1}(X_t - Y_t)$. Assume the assertion is shown up to level $n-1$ and chose two dyadic rationals $s<t$. Using the multiplicative property, we have
	\begin{align*}
		|\X^n_{s,t} - \Y^n_{s,t}| \leq &\sum_{i=0}^{N-1} |\X^n_{\tau_i,\tau_{i+1}} - \Y^n_{\tau_i,\tau_{i+1}}| + \sum_{l=1}^{n-1} \max_{i=1,\ldots,N} |\X^{n-l}_{s,\tau_i}| \sum_{i=0}^{N-1} |\X^l_{\tau_i,\tau_{i+1}} - \Y^l_{\tau_i,\tau_{i+1}}| \\
		&+ \sum_{l=1}^{n-1} \max_{i=1,\ldots,N} |\X^{n-l}_{s,\tau_i} - \Y^{n-l}_{s,\tau_i}| \sum_{i=0}^{N-1} |\Y^l_{\tau_i,\tau_{i+1}}|.
	\end{align*}
	Now we proceed as in step 2, using the estimates from step 1 to step 3 and the induction hypothesis.
\end{proof}

\bigskip

\textsc{Institute for Mathematics, TU Berlin, Strasse des 17.~Juni 136, 10623 Berlin, Germany}. 

Email: riedel@math.tu-berlin.de

\smallskip

\textsc{Mathematical and Oxford-Man Institutes, University of Oxford, 24-29 St.Giles, Oxford, OX1 3LB}. 

Email: xu@maths.ox.ac.uk

\end{document}